\documentclass[reqno,11pt]{amsart}
\makeatletter
\let\origsection=\section 
\def\section{\@ifstar{\origsection*}{\mysection}}
\def\mysection{\@startsection{section}{1}\z@{.7\linespacing\@plus\linespacing}{.5\linespacing}{\normalfont\scshape\centering\S}}
\makeatother

\usepackage{amsmath,amssymb,amsthm}
\usepackage{mathrsfs}
\usepackage{dsfont}
\usepackage{mathabx}\changenotsign
\usepackage{xparse}
\usepackage{microtype}

\usepackage{xcolor}
\usepackage[backref=page]{hyperref}
\hypersetup{%
    colorlinks,
    linkcolor={red!60!black},
    citecolor={green!60!black},
    urlcolor={blue!60!black}
}

\usepackage{varioref}

\usepackage[english,algoruled,vlined,resetcount,linesnumbered]{
algorithm2e}

\usepackage{bookmark}

\usepackage[abbrev,msc-links,backrefs]{amsrefs}
\usepackage{doi}

\renewcommand{\PrintDOI}[1]{\doi{#1}}

\usepackage[T1]{fontenc}
\usepackage{lmodern}

\usepackage[english]{babel}

\usepackage{tikz}
\usetikzlibrary{cd}

\usepackage{fullpage}
\usepackage{setspace}

\usepackage{enumitem}
\newcommand\rmlabel{\upshape({\itshape\roman*\,\/})}

\let\polishlcross=\l
\def\l{\ifmmode\ell\else\polishlcross\fi}

\newcommand\qand{\quad\text{and}\quad}

\renewcommand{\emptyset}{\varnothing}

\makeatletter
\def\moverlay{\mathpalette\mov@rlay}
\def\mov@rlay#1#2{\leavevmode\vtop{%
   \baselineskip\z@skip \lineskiplimit-\maxdimen
   \ialign{\hfil$\m@th#1##$\hfil\cr#2\crcr}}}
\newcommand{\charfusion}[3][\mathord]{
    #1{\ifx#1\mathop\vphantom{#2}\fi
        \mathpalette\mov@rlay{#2\cr#3}
      }
    \ifx#1\mathop\expandafter\displaylimits\fi}
\makeatother

\newcommand{\sr}{\hat{r}}
\newcommand{\eps}{\varepsilon}

\newtheoremstyle{case}{}{}{\normalfont}{}{\itshape}{:}{ }{}

  \newtheorem{theorem}{Theorem}
  \newtheorem{lem}[theorem]{Lemma}
  \newtheorem{thm}[theorem]{Theorem}
  \newtheorem{cor}[theorem]{Corollary}
  
  \newtheorem*{remark*}{Remark}

  \newtheorem{claim}[theorem]{Claim}
  \newtheorem{prop}[theorem]{Proposition}
  \newtheorem{defn}[theorem]{Definition}

\newtheoremstyle{case}{}{}{\normalfont}{}{\itshape}{\normalfont:}{ }{}

\theoremstyle{case}

\numberwithin{equation}{section}
\numberwithin{theorem}{section}

\def\blue{\text{\rm blue}}
\def\grey{\text{\rm grey}}

\let\epsilon\varepsilon
\let\:\colon
\let\subset\subseteq

\def\({\left(}
\def\){\right)}
\def\[{\left[}
\def\]{\right]}
\let\<\langle
\let\>\rangle

\newcommand{\EE}{\mathbb{E}}

\newcommand{\PP}{\mathbb{P}}
\usepackage{datetime}
\usepackage{lineno}
\newcommand*\patchAmsMathEnvironmentForLineno[1]{%
\expandafter\let\csname old#1\expandafter\endcsname\csname #1\endcsname
\expandafter\let\csname oldend#1\expandafter\endcsname\csname end#1\endcsname
\renewenvironment{#1}%
{\linenomath\csname old#1\endcsname}%
{\csname oldend#1\endcsname\endlinenomath}}%
\newcommand*\patchBothAmsMathEnvironmentsForLineno[1]{%
\patchAmsMathEnvironmentForLineno{#1}%
\patchAmsMathEnvironmentForLineno{#1*}}%
\AtBeginDocument{%
\patchBothAmsMathEnvironmentsForLineno{equation}%
\patchBothAmsMathEnvironmentsForLineno{align}%
\patchBothAmsMathEnvironmentsForLineno{flalign}%
\patchBothAmsMathEnvironmentsForLineno{alignat}%
\patchBothAmsMathEnvironmentsForLineno{gather}%
\patchBothAmsMathEnvironmentsForLineno{multline}%
}

\def\cP{\mathcal{P}}

\newcommand{\oldqed}{}
\def\endofFact{\hfill\scalebox{.6}{$\Box$}}
\newenvironment{claimproof}[1][Proof]{
  \renewcommand{\oldqed}{\qedsymbol}
  \renewcommand{\qedsymbol}{\endofFact}
  \begin{proof}[#1]
}{
  \end{proof}
  \renewcommand{\qedsymbol}{\oldqed}
}

\begin{document}
\onehalfspace

\title{The multicolour size-Ramsey number of powers of paths}

\shortdate
\yyyymmdddate
\settimeformat{ampmtime}
\date{\today, \currenttime}

\author[
Han
\and Jenssen
\and Kohayakawa
\and Mota
\and Roberts
]
{
Jie Han
\and Matthew Jenssen
\and Yoshiharu Kohayakawa
\and Guilherme Oliveira Mota
\and Barnaby Roberts
}

\shortdate
\yyyymmdddate
\settimeformat{ampmtime}
\date{\today, \currenttime}

\address{Department of Mathematics, University of Rhode Island, Kingston, RI, USA}
\email{jie\_han@uri.edu}

\address{Mathematical Institute, University of Oxford, Oxford, United Kingdom}
\email{jenssen@maths.ox.ac.uk}

\address{Instituto de Matem\'atica e Estat\'{\i}stica, Universidade de
   S\~ao Paulo, S\~ao Paulo, Brazil}
\email{yoshi@ime.usp.br}

\address{Centro de Matem\'atica, Computa\c c\~ao e Cogni\c c\~ao, Universidade Federal do ABC, Santo Andr\'e, Brazil}
\email{g.mota@ufabc.edu.br}

\address{Department of Mathematics, London School of Economics, London, United Kingdom}
\email{roberts.barnaby@gmail.com}

\thanks{%
  The third author was partially supported by FAPESP
  (Proc.~2013/03447-6) and by CNPq (Proc.~459335/2014-6,
  310974/2013-5).
  The fourth author was supported by FAPESP (Proc.~2018/04876-1) and
  by CNPq (Proc.~304733/2017-2).
  This study was financed in part by the Coordena\c c\~ao de
  Aperfei\c coamento de Pessoal de N\'ivel Superior, Brasil (CAPES),
  Finance Code 001}

\begin{abstract}
  Given a positive integer $s$, a graph $G$ is \emph{$s$-Ramsey} for a
  graph $H$, denoted $G\rightarrow (H)_s$, if every $s$-colouring of
  the edges of $G$ contains a monochromatic copy of $H$.  The
  \emph{{$s$-colour} size-Ramsey number} $\sr_s(H)$ of a graph $H$ is
  defined to be $\sr_s(H)=\min\{|E(G)|\colon G\rightarrow (H)_s\}$.
  We prove that, for all positive integers~$k$ and $s$, we have
  $\sr_s(P_n^k)=O(n)$, where~$P_n^k$ is the $k$th power of the
  $n$-vertex path $P_n$.
\end{abstract}

\maketitle

\section{Introduction}
\label{sec:intro}

Given a graph $H$ and a positive integer $s$, the \emph{Ramsey number}~$r_s(H)$ of~$H$ is the smallest number~$N$ such that any $s$-colouring of $E(K_N)$ contains a monochromatic copy of $H$.
The existence, for any graph $H$, of such an $N$ follows from the celebrated Ramsey Theorem~\cite{Ra}.
More generally, given a graph $G$, one can ask if any $s$-colouring of the edges of $G$ contains a monochromatic copy of~$H$.
In view of this, denote by $G\rightarrow (H)_s$ the property that any $s$-colouring of $E(G)$ contains a monochromatic copy of~$H$.
For recent advances in Ramsey theory we refer the reader to the survey of Conlon, Fox and Sudakov~\cite{CoFoSu15}.
One variant of Ramsey problems that has attracted the attention of researchers started with a problem proposed by Erd\H{o}s, Faudree, Rousseau and Schelp~\cite{ErFaRoSc78}.
They posed the question of finding the minimal number of edges a graph
$G$ can have with the property that $G\rightarrow (H)_2$.
Define the \emph{$s$-colour size-Ramsey number} $\sr_s(H)$ of a graph $H$ to be
\begin{equation*}
\sr_s(H):=\min\{|E(G)|\colon G\rightarrow (H)_s\}.
\end{equation*}

Note that we have the trivial upper bound $\sr_s(H)\leq \binom{r_s(H)}{2}$ for any graph $H$ and Chv\'atal~(see~\cite{ErFaRoSc78}) observed that this is in fact tight when $H$ is a clique. 
An early object of study in the exploration of size Ramsey numbers was the graph $P_n$, the path on $n$ vertices. 
It is easy to show (e.g., using the Erd\H{o}s--Gallai Theorem~\cite{erdHos1959maximal}) that $r_s(P_n)\le s n$ and so for $s$ fixed we have that $\sr_s(P_n)=O(n^2)$. 
Note that we trivially also have that $\sr_s(P_n)=\Omega(n)$. Focusing on the case of $2$ colours, Erd\H{o}s~\cite{erdHos1981combinatorial} offered $\textdollar100$ for a proof or disproof of the assertion that $n\ll \sr_2(P_n)\ll n^2$.
A couple of years later Beck~\cite{Be83} claimed this prize by disproving the assertion and showing (in a probabilistic way) that $\sr_2(P_n)=O(n)$.
In fact Beck proved the following density version: for any $\eps>0$ there exists a graph $G$ with $O(n)$ vertices and edges such that any subgraph of $G$ containing an $\eps$-proportion of the vertices of $G$ must contain a copy of $P_n$.
An explicit construction of a graph with the same properties was given by Alon and Chung~\cite{AlCh88}.
In particular, the existence of such graphs shows $\sr_s(P_n)=O(n)$ for any fixed number of colours~$s$.
It has recently been shown by Dudek and Pra{\l}at~\cite{dudek2017some} and Krivelevich~\cite{krivelevich2017long} (see also \cite{dudek2018note}) that in fact
\begin{equation*}
c s^2 n\leq\sr_s(P_n)\leq C s^2(\log s) n
\end{equation*}
for some absolute constants $c$ and $C$.
For the case $s=2$, after many successive improvements (\cite{Be83, bollobas1986extremal, dudek2017some, YouLin} for the lower bound, \cite{Be83, bollobas1998random, DuPr15, letzter16:_path_ramsey, dudek2017some} for the upper bound) the state of the art is 
\begin{equation*}
3n - 7\leq\sr_2(P_n)\leq 74n\, .
\end{equation*}

In light of these results, one might wonder whether $\sr_2(G)$ is linear for all graphs $G$ with bounded maximum degree and indeed Beck~\cite{Be90} proposed this very question. 
This question was subsequently answered in the negative by R\"odl and Szemer\'edi~\cite{RoSz00} who proved that there exists an $n$-vertex graph $H$
with $\Delta(H)\leq 3$  such that $\sr_2(H)=\Omega(n\log ^{1/60}n)$.
On the other hand, in~\cite{KoRoScSz11} it was proved that for any $n$-vertex graph $H$ with maximum degree~$\Delta$ we have
\begin{equation*}
  \sr(H)\leq cn^{2-1/\Delta}\log^{1/\Delta}n,
\end{equation*}
where $c$ is a constant that depends only on $\Delta$.

For more results on the size-Ramsey number of bounded degree graphs see~\cite{De12,FrPi87,HaKo95,HaKoLu95,Ke93, kohayakawa07+:_ramsey}.
See also~\cite{%
  ben-eliezer12:_Ramsey,
  reimer02:_ramsey
} for related results in a digraph setting.

Given a graph $H$ with $n$ vertices and an integer $k\geq 2$, the $k$th power of~$H$, denoted $H^k$, is the graph with vertex set $V(H)$ where distinct vertices $u$ and $v$ are adjacent if and only if the distance between them in $H$ is at most~$k$.
In~\cite{2colourSizeRamsey}, it was proved that 
powers of paths have linear $2$-colour size-Ramsey numbers.
Formally,   
\begin{thm}[\cite{2colourSizeRamsey}]\label{thm:2colours}
  For any integer $k\geq 2$,
  \begin{equation}
    \label{eq:thm_main}
    \sr_2(P_n^k)=O(n).
  \end{equation}
\end{thm}

In our main result, Theorem~\ref{thm:main} below, we extend Theorem~\ref{thm:2colours} to an arbitrary number of colours, proving that the multicolour size-Ramsey number of powers of paths is linear, which verifies~\cite[Conjecture~4.2]{2colourSizeRamsey}.

\begin{thm}\label{thm:main}
  For any integers $s$ and $k\geq 1$,
\begin{equation*}
    \sr_s(P_n^k)=O(n).
    \end{equation*}
\end{thm}

Like Beck's result for paths~\cite{Be83}, our proof will be based on a probabilistic construction.
But since for $k\geq2$ the graph $P_n^k$ contains triangles, we cannot hope for a density analogue of Theorem~\ref{thm:main} as any graph $G$ has a subgraph of positive relative density containing no triangles (e.g.~choose a max-cut of $G$).

Note that since $C_n^k\subset P_n^{2k}$ we have the following immediate corollary.
\begin{cor}
For any integers $s$ and $k\geq 1$,
\begin{equation*}
    \sr_s(C_n^k)=O(n).
    \end{equation*}
\end{cor}

There are significant obstacles to overcome in order to generalise the two colour result of Theorem~\ref{thm:2colours} to an arbitrary number of colours.
For our proof, rather than searching for monochromatic copies of the $P_n^k$ directly, we proceed by finding a sequence of nested subgraphs using fewer and fewer colours, each of which enjoys certain pseudorandomness properties. This process terminates in a monochromatic pseudorandom graph in which we can embed $P_n^k$ directly. We will phrase this proof in terms of an induction on the number of colours.   

 In Section~\ref{sec:prelim} we provide some auxiliary results needed for the proof of Theorem~\ref{thm:main}, which is given in detail in Section~\ref{sec:main}.
We omit floor and ceiling signs when they are not essential.

\section{Preliminaries}\label{sec:prelim}
In this section we present some simple facts and then discuss the tools required in the proof of our main result.
Some notation and definitions that will be used throughout the paper are also introduced.
We start by defining standard notions of density in a graph.
For a graph $G$ and a subset $S\subseteq V(G)$ we define the \emph{density} of $S$ to be
\begin{equation*}
d_G(S):=e(G[S])\binom{|S|}{2}^{-1}\, ,
\end{equation*}
where $e(G[S])$ denotes the number of edges in the subgraph of $G$ induced by $S$.
Given disjoint subsets $X$, $Y\subseteq V(G)$, we define the \emph{density between} $X$ \emph{and} $Y$ to be
\begin{equation*}
d_G(X,Y):=\frac{e_G(X,Y)}{|X||Y|}\,,
\end{equation*}
where $e_G(X,Y)$ denotes the number of edges between $X$ and $Y$.
We omit the subscripts in the above notation when the graph $G$ is clear from the context. 

The following basic statement will be useful.

\begin{prop}\label{prop:density}
Let $G$ be an $n$-vertex graph and let $\alpha$, $\eps$, $f>0$. 
Suppose that for all disjoint subsets $X$, $Y\subseteq V(G)$ with $|X|=|Y|=\alpha n$ we have $d_G(X,Y)=(1\pm\eps)f$.
Then for all disjoint $U$, $W\subseteq V(G)$ with $|U|$, $|W|\geq \alpha n$, we have
\begin{equation*}
d_G(U,W)=(1\pm\eps)f.
\end{equation*}
Moreover if $S\subseteq V(G)$ and $|S|\geq 2\alpha n$, then
\begin{equation*}
d_G(S)= (1\pm\eps)f.
\end{equation*}
\end{prop}

We will also need Proposition~\ref{cl:edgeboost} below in the proof of Theorem~\ref{thm:main}.

\begin{prop}\label{cl:edgeboost}
Let $2\mu\leq \beta\leq \alpha$. If $G$ is an $\alpha n$-vertex graph such that there is an edge between any two disjoint sets of $\mu n$ vertices, then between any two disjoint sets of $\beta n$ vertices there are at least $(\beta^2/2\mu)n$ edges.
\end{prop}
\begin{proof}
Let $X$, $Y\subseteq V(G)$ be two disjoint sets with $|X|=|Y|=\beta n$.
Let $X'\subseteq X$ be the subset of $\mu n$ vertices in $X$ of least degree into $Y$.
Note that $|\{y\in Y\colon N(y)\cap X'=\emptyset\}|<\mu n$ by assumption and so $e(X',Y)\geq (\beta-\mu)n\geq \beta n/2$. 
It follows that the average degree of vertices in $X'$ is at least $\beta/2\mu$ and so, by the choice of $X'$, 
\begin{equation*}
e(X,Y)\geq \frac{\beta^2n}{2\mu},
\end{equation*}
as required.
\end{proof}

We use the following version of Chernoff's inequality, which is an immediate corollary of \cite[Theorem 2.1]{janson2011random}.

\begin{thm}[Chernoff's inequality]\label{thm:chernoff}
Let $0<\eps\leq 3/2$.
If $X$ is a sum of independent Bernoulli random variables then
\begin{equation*}
\PP(|X-\EE[X]| > \eps \EE[X]) \leq 2\cdot e^{-(\eps^2/3)\EE[X]}\,.
\end{equation*}
\end{thm}

Theorem~\ref{thm:alexey} below will be applied several times in our proof.
Given a positive integer $t$, we say that a $t$-partite graph is \emph{balanced} if its $t$ vertex classes have the same size. 

\begin{thm}[{Pokrovskiy~\cite[{Theorem~1.7}]{Po16}}]
  \label{thm:alexey}
  Let $\ell \ge 1$. Suppose that the edges of $K_n$ are coloured with red
  and blue.  Then $K_n$ can be covered by $\ell$ vertex-disjoint blue
  paths and a vertex-disjoint red balanced complete $(\ell+1)$-partite graph.
\end{thm}

We shall also use the classical K\H{o}v\'ari--S\'{o}s--Tur\'{a}n
theorem~\cite{KoSoTu54}, in the following simple form.

\begin{thm}
  \label{thm:kst}
  Let $k\geq1$ and let~$G$ be a balanced bipartite graph with~$x$ vertices in each
  vertex class.  If~$G$ contains no~$K_{2k,2k}$, then~$G$ has at
  most~$4x^{2-1/2k}$ edges.
\end{thm}

We use the following result to find a long path in a graph with pseudorandom properties.

\begin{lem}[{\cite[Lemma~3.5]{2colourSizeRamsey}}]\label{lem:H2}
For every integer $t\geq 1$ and every $\gamma>0$ there exists $d_0>0$ such that the following holds for any $d\ge d_0$.
Let $G$ be a graph with $dn$ vertices such that for every pair of disjoint sets $X,Y\subseteq V(G)$ with $|X|, |Y|\geq \gamma n$ we have $|E_{G}(X,Y)|> 0$.
Then for every family $V_1,\dots,V_{t}\subseteq V(G)$ of pairwise disjoint sets each of size at least $\gamma d n$, there is a path $P_n=(x_1,\dots,x_{n})$ in~$G$ with $x_i \in V_j$ for all~$1\leq i\leq n$, where $j\equiv i\pmod{t}$. 
\end{lem}

\section{Proof of Theorem~\ref{thm:main}}\label{sec:main}

\subsection{Preliminaries and proof strategy}
\label{sec:preliminaries-proof}
We start with a few definitions.

\begin{defn}[Complete blow-ups]\label{def:blow}
Given a graph $H$ and a positive integer $t$, we denote by $H(t)$ the graph obtained by replacing each vertex $v$ in $V(H)$ by a complete graph with $t$ vertices, denoted by $C(v)$, and by adding all edges between $C(u)$ and $C(v)$ for each $uv\in E(H)$.
\end{defn}

\begin{defn}[Sheared complete blow-ups]\label{def:shearedblow}
Given a graph $H$ and a positive integer $t$, we denote by $H\{t\}$ a graph obtained from $H(t)$ by removing a perfect matching from the complete bipartite graph between $C(u)$ and $C(v)$ for each $uv\in E(H)$. 
\end{defn}

Note that, unlike $H(t)$, the graph $H\{t\}$ is not uniquely defined.
We will refer to the cliques $C(u)$ as the \emph{$t$-cliques} of $H(t)$ or $H\{t\}$. 

We say a graph $G$ is \emph{pseudorandom} if it satisfies the properties described in Definition~\ref{def:P} below.
In the proof of our main result (Theorem~\ref{thm:main}) we will prove that either we find the desired monochromatic power of a path in a $s$-coloured pseudorandom graph $G$ or we obtain a sufficiently pseudorandom subgraph $H\subset G$ which is coloured with $s-1$ colours.
In the next definition we describe the properties a graph must have in order to make the induction process work.

\begin{defn}\label{def:P}
For positive numbers $a,b,c,t,\eps,n$, let $\cP(a,b,c,t,\eps,n)$ denote the class of all graphs $G$ with the following properties: 
\begin{enumerate}[label=\rmlabel]
\item $|V(G)|=an$.\label{def:P-1}
\item $\Delta(G)\leq b$.\label{def:P-2}
\item There exists $f_G>0$ such that, for all pairs of disjoint subsets $X$, $Y\subseteq V(G)$ of size $cn$, we have $d(X,Y)=(1\pm\eps)f_G$.\label{def:P-3}
\item $G$ has no cycles of length at most $2t$.\label{def:P-4}
\end{enumerate}
\end{defn}

In Definition~\ref{def:good} below we give a relation between the parameters $a$, $b$, $c$ and $\eps$ which guarantees that the family $\cP(a,b,c,t,\eps,n)$ is non-empty for any positive $t$ and large $n$ (see Lemma~\ref{lem:exist} below).

\begin{defn}\label{def:good}
A quadruple of positive numbers $(a,b,c,\eps)$ is \emph{good} if $a\geq2c+1$, $b \geq 264a^2\eps^{-2}c^{-2}$ and $\eps <1/10$.
\end{defn}

\begin{lem}\label{lem:exist}
If $(a,b,c,\eps)$ is a good quadruple then for any positive integer $t$ there exists a graph $G\in\cP(a,b,c,t,\eps,n)$ for $n$ sufficiently large.
\end{lem}

\begin{proof}
Let $(a,b,c,\eps)$ be a good quadruple and $t$ a positive integer.
Let $G^* = G(2an, p)$ be the binomial random graph with $2an$ vertices and edge probability $p = 60a/(\eps^2 c^2n)$.
Using Chernoff's inequality (Theorem~\ref{thm:chernoff}), we see that the probability that there is a pair of sets $X$ and $Y$ with size~$cn$ such that $e_{G^*}(X,Y)=(1\pm\eps/2)pc^2n^2$ does \emph{not} hold is at most
\begin{equation}\label{eq:prob}
\binom{2an}{cn}^2 \cdot 2 \cdot e^{-(\eps/2)^2 pc^2 n^2/3}<\frac{1}{2},
\end{equation}
where the above inequality follows from the choice of $p$ and the fact that $n$ is sufficiently large.

We now calculate the expected number $C_{\leq 2t}$ of cycles of length at most $2t$ in $G^*$.
Let $C_i$ be the number of cycles of length exactly $i$, for any integer $i\geq 3$.
We have
\begin{equation*}
\mathbb{E}(C_{\leq 2t}) 	=		\sum_{i=3}^{2t} \mathbb{E}(C_i)
					= 		\sum_{i=3}^{2t} {2an\choose i}\frac{(i-1)!}{2}\, p^i
					\leq 	\sum_{i=3}^{2t} \left(\frac{11a}{\eps c}\right)^{2i}
					\leq	2t\left(\frac{11a}{\eps c}\right)^{4t}.
\end{equation*}
Then, from Markov's inequality, we have
\begin{equation}\label{eq:prob2}
\mathbb{P}\big(	C_{\leq 2t}\geq  4t\big(11a/(\eps c)\big)^{4t} \big)		\leq \frac{1}{2}.
\end{equation}
Since the sum of the probabilities in~\eqref{eq:prob} and~\eqref{eq:prob2} is smaller than~1, there exists a graph $G'$ with $2an$ vertices that contains less than $4t\big(11a/(\eps c)\big)^{4t}$ cycles of length at most $2t$, and for every pair of sets $X$ and $Y$ of size $cn$ we have $e_{G'}(X,Y)=(1\pm\eps/2)pc^2n^2$.
Then, by removing fewer than $4t\big(11a/(\eps c)\big)^{4t}$ edges from $G'$ we can destroy all cycles of length at most $2t$.
Since only a constant number of edges were removed and $n$ is sufficiently large, we obtain a graph $G''$ with no cycles of length at most $2t$ such that $e_{G''}(X,Y)=(1\pm\eps)pc^2n^2$ for every pair of sets $X$ and $Y$ of size $cn$.

By Proposition~\ref{prop:density} and the fact that $a>2c$, we see that the density of the whole graph $G''$ is at most $(1+\eps)p$, which implies that
\begin{equation*}
|E(G'')|\leq (1+\eps)p\binom{2an}{2} \leq \frac{264 a^3 }{\eps^2 c^2}n.
\end{equation*}
Repeatedly removing vertices of highest degree in $G''$ until $an$ vertices are left, we obtain a subgraph $G\subset G''$ such that $\Delta(G)\leq 264 a^2\eps^{-2}c^{-2} \leq b$, as otherwise we would have deleted more than $|E(G'')|$ edges.
Since deleting vertices preserves the property that pairs of subsets of size~$cn$ have the correct density, this completes the proof of the proposition.
\end{proof}

Our strategy for proving Theorem~\ref{thm:main} consists in proving
the following.

\begin{prop}\label{prop:induction}
For all positive integers $k$ and $s$, there exist positive integers $r$ and $t$ and a good quadruple $(a,b,c,\eps)$ such that if $n$ is sufficiently large and $G\in \cP(a,b,c,t,\eps,n)$ then 
\begin{equation*}
G^r\{t\}\rightarrow (P_n^k)_s.
\end{equation*}
\end{prop}

Note that for any $G\in \cP(a,b,c,t,\eps,n)$ we have $\Delta(G)\leq b$ by definition.
It follows that, for every fixed $r$, we have $|E(G^r)|=O(n)$ and so
 \begin{equation*}
 |E(G^r\{t\})|=|E(G^r)|(t^2-t)+an\binom{t}{2}=O(n).
 \end{equation*} 
Therefore, Theorem~\ref{thm:main} follows directly from Proposition~\ref{prop:induction} and Lemma~\ref{lem:exist}. 

It remains to prove Proposition~\ref{prop:induction}.  We apply
induction on~$s$, the number of colours.

\begin{lem}[Induction Step]\label{lem:indstep}
For all positive integers $k$, $s$, $r$ and $t\geq 2$ and a good quadruple $(a,b,c,\eps)$, there exist positive integers $R$ and $T$ and a good quadruple $(A,B, C,\delta)$ such that the following holds.
If $n$ is sufficiently large and $G\in\cP(A,B,C,T,\delta,n)$, then in any $s$-colouring of $E(G^R\{T\})$ either 
\begin{enumerate}[label=\rmlabel]
\item there exists a monochromatic copy of $P_n^k$, or\label{def:indstep-1}
\item there exists $H\in\cP(a,b,c,t,\eps,n)$ such that $H^r\{t\} \subset G^R\{T\}$ and this copy of $H^r\{t\}$ uses at most $s-1$ colours.\label{def:indstep-2}
\end{enumerate}
\end{lem}

\begin{lem}[Base Case]\label{lem:base}
For every positive integer $k$, there exist positive integers $r$ and $t\geq 2$ and a good quadruple $(a,b,c,\eps)$ such that, if $n$ is sufficiently large and $G\in\cP(a,b,c,t,\eps,n)$, then $G^r\{t\}$ contains a copy of $P_n^k$.
\end{lem}

We now prove that Lemmas~\ref{lem:indstep} and~\ref{lem:base} imply Proposition~\ref{prop:induction}.

\begin{proof}[Proof of Proposition~\ref{prop:induction}]
Let $k$ be a positive integer.
For $s=1$ the conclusion of the proposition follows from the base case, Lemma~\ref{lem:base}.
Fix any $s \geq 2$ and assume the result holds for $s-1$ colours.
Let $a = a(s-1)$, $b=b(s-1)$, $c=c(s-1)$, $\eps = \eps(s-1)$, $r = r(s-1)$, and $t = t(s-1)$ be the constants given by the inductive hypothesis.
We will obtain parameters $A$, $B$, $C$, $\delta$, $R$ and $T$ such that $(A,B,C,\delta)$ is a good quadruple and such that, for any graph $G\in\cP(A,B,C,T,\delta,n)$, any $s$-colouring of $G^R\{T\}$ contains a monochromatic copy of $P_n^k$.

Let $R$ and $T$ and a good quadruple $(A,B,C,\delta)$ be given by Lemma~\ref{lem:indstep} applied with $s$, $r$, $t$ and $(a,b,c,\eps)$.
Then, Lemma~\ref{lem:indstep} implies that, for any $G\in\cP(A,B,C,T,\delta,n)$, every $s$-colouring of the edges of $G^R\{T\}$ either~\ref{def:indstep-1} contains a monochromatic copy of $P_n^k$, or~\ref{def:indstep-2} there is a graph $H\in\cP(a,b,c,t,\eps,n)$ such that $H^r\{t\} \subset G^R\{T\}$ and the edges of $H^r\{t\}$ are coloured with at most $s-1$ colours.
In the former case we are done and in the latter case we apply induction to find a monochromatic copy of $P_n^k$ in $H^r\{t\}$ and hence also in $G^R\{T\}$.
\end{proof}

It remains to prove Lemmas~\ref{lem:indstep} and~\ref{lem:base}.

\subsection{Proofs of Lemmas~\ref{lem:indstep} and \ref{lem:base}}
\label{sec:proofs-lemmas}

Let us begin by establishing Lemma~\ref{lem:base}, the base case of
our induction.

\begin{proof}[Proof of Lemma~\ref{lem:base}]
Let $k$ be a positive integer and let $r=k$ and $t=k+1$.
Let $(a,b,c,\eps)$ be a good quadruple and let 
\begin{equation*}
G\in\cP(a,b,c,t,\eps,n),
\end{equation*}
where $n$ is sufficiently large.

Applying Theorem~\ref{thm:alexey} with $\ell=1$, considering the edges of $G$ as blue edges and the non-edges of $G$ as red, we may partition the vertex set of $G$
with a path $P$ and two sets $X$ and $Y$ of the same size such that there is no edge between them.
Since $G\in\cP(a,b,c,t,\eps,n)$
we must have that $|X|=|Y|<c n$.
Since $|V(G)|=an$, and $a\geq 2c+1$, we have $|P|>(a-2c)n\geq n$. 
We may therefore fix a path $P_n=(v_1,\ldots, v_n)$ in $G$.
By the definition of the $k$th power of a path, we have $P_n^k\subseteq G^k$.
We will finish the proof by embedding a copy of $P_n^k$ in $G^k\{k+1\}$.

Given a vertex $v\in V(G)$, $C(v)$ denotes the $(k+1)$-clique in $G^k\{k+1\}$ that corresponds to $v$.
Suppose that for some $1\leq\ell<k$, 
we have embedded a copy of $P_\ell^k$ in $G^k\{k+1\}$ where $V(P_\ell^k)=\{w_1,\ldots,w_\ell\}$ and $w_i\in C(v_i)$ for $i=1,\ldots,\ell$.
By the definition of $G^k\{k+1\}$ (Definition~\ref{def:shearedblow}), each of the $k$ vertices $w_{\ell-(k-1)},\ldots,w_\ell$ is adjacent to all but one vertex of $C(v_{\ell+1})$.
Since $|C(v_{\ell+1})|=k+1$, we may thus find a vertex $w_{\ell+1}\in C(v_{\ell+1})$ such that $w_{\ell+1}$ is adjacent to each of $w_{\ell-(k-1)},\ldots,w_\ell$.
Then, the vertices $w_1,\ldots,w_{\ell+1}$ induce a copy of $P_{\ell+1}^k$ in $G^k\{k+1\}$ where $w_i\in C(v_i)$ for $i=1,\ldots,\ell+1$.
Therefore, starting with any vertex $w_1$ in $C(v_1)$, we may obtain a copy of $P_n^k$ in $G^k\{k+1\}$ inductively, which proves the lemma.
\end{proof}

It remains to prove Lemma~\ref{lem:indstep}.
First we give a sketch of the proof.
Let $G$ be a pseudorandom graph (see Definition~\ref{def:P}) and consider an $s$-colouring $\chi$ of a sheared complete blow-up $G^R\{T\}$ of a power of $G$.
We shall prove that either $G^R\{T\}$ contains the desired monochromatic power of a path, or there exists a pseudorandom graph $H$ such that a sheared complete blow-up $H^r\{t\}	$ of a power of~$H$ appears as a subgraph of $G^R\{T\}$ and is coloured with $s-1$ colours.
(We also need that~$H$ should be pseudorandom with suitable
parameters, but let us forget such technicalities for now.)  The proof
will involve several graphs and, to keep track of the relationship among
these graphs, the reader may find it useful to refer to
Figure~\ref{fig:graphs}.

\begin{figure}[h]
  \centering
  \begin{tikzcd}
    P_n^k \arrow[rr, hook, "\text{blue}"]	& &G^R\{T\}, \chi_s \arrow[d, two heads] \\
    P_n \arrow[r, hook, "\text{blue}"]  \arrow[u, Rightarrow, "\text{Claim~\ref{bluepn} }"]& J,\chi'_{\blue,\grey} \arrow[r, hook, "\text{induced}"] & G^R
  \end{tikzcd}
  \qquad\vspace{1.2cm}

    \begin{tikzcd}
    &  & H^r\{t\}
    \arrow[ld, two heads, end anchor={[xshift=-0.5ex]}] \arrow[drr, hook, "\text{no blue edges (Claim~\ref{claim:LLL})}"] \arrow[dd, hook, "\text{grey (Claim~\ref{claim:Hblow})}"]& & \\
     & H'  \supseteq  H''\supseteq H \in \mathcal{P} &  &  & G^R\{T\},\chi_s \arrow[d, two heads] \\
    & J' \arrow[d, dashed, "V(J') = V(J'')"] \arrow[r, hook] & J,\chi'_{\blue,\grey} \arrow[rr, hook, "\text{induced}"] & & G^R\\
    &  J'' \arrow[uu, dashed, start anchor={[xshift=-0.7ex, yshift=-0.2ex]},
end anchor={[xshift=-8.1ex,yshift=0.4ex]}]  \arrow[rrr, hook, "\text{induced}"] &  & & G \arrow[u, hook] 
  \end{tikzcd}
  
  \caption{The several graphs appearing in the proof of
    Lemma~\ref{lem:indstep}. The two diagrams correspond to the two
    alternatives in the proof (that is, whether or not~$J$ contains a
    blue~$P_n$). The hooked arrows are inclusions (to be more precise,
    subgraph or induced subgraph containment) and the two-headed
    arrows are the natural projections.  A dashed arrow from~$X$
    to~$Y$ means that~$Y$ is defined in terms of~$X$.  Finally,
    $\chi_s$~is the $s$-colouring~$\chi$ given to~$G^R\{T\}$ by our
    adversary, and~$\chi_{\blue,\grey}'$ is an auxiliary
    $\{\blue,\grey\}$-colouring~$\chi'$ of~$J'$, as discussed in the
    proof sketch.}
  \label{fig:graphs}
\end{figure}
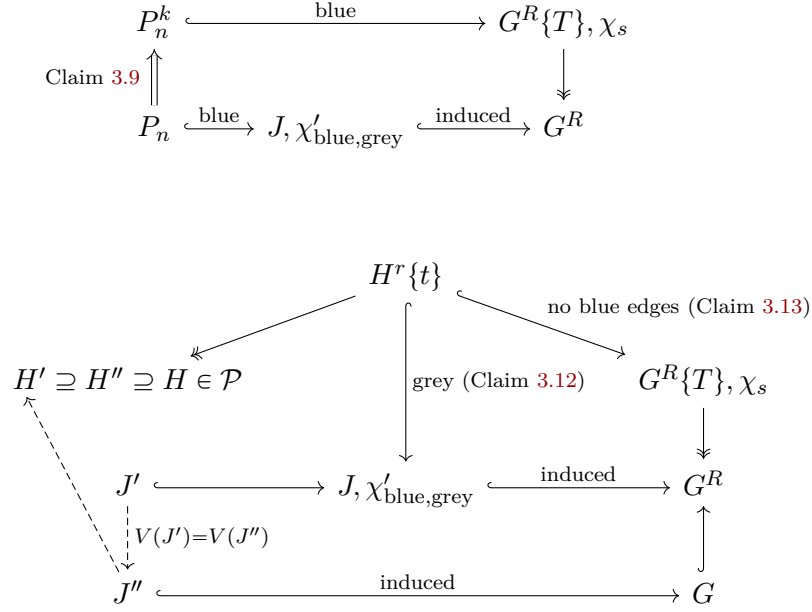

We start by noticing that, if $T$ is large, each $T$-clique $C(v)$ of $G^R\{T\}$ contains a large monochromatic clique by Ramsey's theorem.
Also, at least a $1/s$ fraction of these monochromatic cliques have
the same colour, say blue.  Let us write~$B(v)\subset C(v)$ for such blue cliques.
Then we consider $J\subset G^R$ induced by the
vertices~$v$ such that~$C(v)$ contains a monochromatic blue
clique~$B(v)$.
We next consider an auxiliary $\{\text{blue, grey}\}$-colouring $\chi'$
of the edges of~$J$:
the edge $\{u,v\}\in E(J)$ is coloured blue if there is a blue copy
of~$K_{2k,2k}$ under~$\chi$ in the
bipartite graph induced between the blue cliques $B(u)$ and $B(v)$
in $G^{R}\{T\}$, and colour~$\{u,v\}$ grey otherwise. 

We next prove that if $J$ contains a blue $P_n$, then~$G^R\{T\}$
contains a blue $P^k_n$ (see Claim~\ref{bluepn}).  Then, assuming that
$J$ does not contain a blue~$P_n$, we obtain a multipartite graph
$J'\subset J$ with no blue edges and we prove that the subgraph~$J''$
of~$G$ induced by~$V(J')$ contains $2an$ grey paths
$Q_1, \ldots, Q_{2an}$ each with $t$ vertices.  We then consider an
auxiliary graph $H'$ on $V(H')=\{Q_1, \ldots, Q_{2an}\}$ with
$Q_iQ_j\in E(H')$ if and only if there is an edge between the vertex
sets of $Q_i$ and $Q_j$ in~$G$.  We obtain a sparse subgraph
$H''\subset H'$ by choosing edges of $H'$ uniformly at random with a
suitable probability~$p$.  Then, successively removing vertices of
high degree, we obtain a graph $H\subset H''$ which satisfies
$\cP(a,b,c,t,\eps,n)$ (Claims~\ref{claim:Hp} and~\ref{cl:new}).  It
will remain to find a copy of $H^r\{t\}$ in $G^R\{T\}$ with no blue
edges.  It turns out that it suffices to find a grey copy of
$H^r\{t\}$ in~$J$ (Claim~\ref{claim:LLL}).  Finally, the existence of
such a copy of $H^r\{t\}$ in~$J$ is guaranteed by
Claim~\ref{claim:Hblow}.  In what follows, we expand on this outline
and give a detailed proof.


\begin{proof}[Proof of Lemma~\ref{lem:indstep}]
Let $k$, $s$, $r$ and $t\geq 2$ be positive integers and let $(a,b,c,\eps)$ be a good quadruple.
Let $d_0$ be obtained from Lemma~\ref{lem:H2} applied with $t$ and $\gamma:=1/(2t)$.
Put $T'=b^{2rk}t^{2k}$ and define constants $A$, $B$, $C$, $R$, $T$ and $\delta$ as follows.
\begin{equation}
    \begin{gathered}
       A = 2d_0(a+1)st,\,\,\,\,\,\,\, C = \min\left\{\frac{1}{2st},\frac{\eps^2 c^2}{240a}\right\},\,\,\,\,\,\,\,R=tr,\,\,\,\,\,\,\,\,\delta=\frac{\eps}{2},\\
       B=\frac{264 A^2}{\delta^2C^2} \qand T=s^{sT'}. 
    \end{gathered}\label{eq:constants}
\end{equation}

Note that $(A,B,C,\delta)$ is a good quadruple.
Suppose $G\in\cP(A,B,C,T,\delta,n)$ for a sufficiently large~$n$ (such a~$G$ exists by Lemma~\ref{lem:exist}) and
let $\chi\colon E(G^R\{T\})\to [s]$ be an $s$-colouring of the edges of $G^R\{T\}$. 

Using the bound $r_s(K_m)\leq s^{sm}$ on the multicolour diagonal Ramsey numbers  (see, for example,~\cite{CoFoSu15}),
we see that each of the $T$-cliques in $G^R\{T\}$ contains a monochromatic clique on precisely $T'$ vertices. 
Without loss of generality, at least a $1/s$ fraction of these monochromatic subcliques are blue. 
Thus, we have a collection of blue cliques $\{B(w): w\in W\subseteq V(G)\}$ for some subset $W\subseteq V(G)$ where 
\begin{equation}\label{eq:Wsize}
|W|\geq \frac{|V(G)|}{s}=\frac{An}{s}.
\end{equation}

Consider the graph
$J:=(G[W])^R$
and give the following $2$-colouring~$\chi'$ to the edges of~$J$: colour the edge $\{u,v\}\in E(J)$ blue if there is a blue copy of $K_{2k,2k}$ under $\chi$ in the bipartite graph between the blue cliques $B(u)$ and $B(v)$ in $G^{R}\{T\}$, and colour the edge $\{u,v\}$ grey otherwise.

The next claim shows that the existence of a blue $P_n$ in $J$ (under the colouring $\chi'$) implies the existence of a blue $P^k_n$ in $G^{R}\{T\}$.

\begin{claim}\label{bluepn}
If $J$ contains a blue copy of $P_n$ under $\chi'$, then $G^{R}\{T\}$ contains a blue copy of $P_n^k$ under~$\chi$.
\end{claim}

\begin{claimproof}
Suppose that
  $(w_1,\dots,w_n)$ is a blue copy of~$P_n$ in $J$ under $\chi'$.
  Then, for every $1 \le i \le n-1$, 
  there is a blue copy of $K_{2k,2k}$ between $B(w_i)$ and $B(w_{i+1})$ in $G^{R}\{T\}$, with,
  say, vertex classes $X_i \subseteq B(w_i)$ and $Y_{i+1}\subseteq B(w_{i+1})$. 
  As $|X_i|=|Y_i|=2k$ for all
  $2\leq i\leq n-1$, one can obtain sets $X'_i \subseteq X_i$ and
  $Y'_i\subseteq Y_i$ such that $|X'_i| = |Y'_i| = k$ and
  $X'_i \cap Y_i' = \emptyset$ for all $2 \le i \le n-1$.
  Let~$X_1'=X_1$ and~$Y_n'=Y_n$.

  We will show that the set
  $U := \bigcup_{i=1}^{n-1}X'_i \cup \bigcup_{i=2}^{n}Y'_i$  spans a blue copy of~$P_{2kn}^k$ in~$G^{R}\{T\}$.  Note first
  that~$|U|=2k+2k(n-2)+2k=2kn$.  Let $u_1,\dots, u_{2kn}$ be an
  ordering of $U$ such that, for each~$i$, every vertex in $X_i'$
  comes before any vertex in $Y_{i+1}'$ and after every vertex in
  $Y_i'$.  By the definition of the sets $X_i'$ and $Y_i'$, each vertex~$u_j$ is
  adjacent in blue to $\{u_{j'}\in U\colon 1 \le |j-j'| \le k \}$. 
  Then, $U$~contains a blue copy of~$P_{2kn}^k$, which contains a blue copy of $P_n^k$ as claimed.
\end{claimproof}

By Theorem~\ref{thm:alexey} (with blue edges of $J$ as one colour class and the rest of the pairs as the other colour class), we may partition the vertex set of $J$ into at most $t-1$ vertex-disjoint blue paths and a vertex-disjoint grey balanced (not necessarily complete) $t$-partite graph $J'$ whose vertex classes we denote by $V_1, \dots, V_t$.

By Claim~\ref{bluepn}, a blue copy of $P_n$ in $J$ would finish the proof, so we may assume that no blue path in $J$ has length $n$.
Recalling that $A = 2d_0(a+1)st$, we have
\begin{equation}\label{eq:Jbig}
|V(J')|\geq|V(J)|-(t-1)n>  \frac{An}{s}- tn\geq d_0 \cdot 2atn\, ,
\end{equation}
where for the penultimate inequality we used \eqref{eq:Wsize} recalling that $V(J)=W$. 
Moreover, since $J'$ is balanced, we have 
\begin{equation}\label{eq:Vbig}
|V_i|\ge |V(J')|/(2t)\, .
\end{equation} 
Let $J'':=G[V(J')]$,
and note that whereas $J$ and $J'$ are subgraphs of $G^R$, $J''$ is a subgraph of~$G$.
 Since $J''$ is an induced subgraph of $G$ and $G\in\cP(A,B,C,T,\delta,n)$, we know from~\ref{def:P-3} that
 for every pair of disjoint sets $X,Y\subset V(J'')$ with $|X|,|Y|\geq C n$ we have $|E_{J''}(X,Y)|\geq 1$.


Note that $V(J'')=V(J')$ and recall that $d_0$ is the constant obtained from Lemma~\ref{lem:H2} applied with $t$ and $\gamma=1/(2t)$.
By \eqref{eq:Jbig}, \eqref{eq:Vbig} and the fact that $C\le a$ (so that $Cn\le \gamma\cdot 2atn$), we may apply Lemma~\ref{lem:H2} to conclude that $J''$ contains a path $P_{2atn}$ on $(x_1,\dots,x_{2atn})$ with $x_i \in V_j$ for all~$i$, where $j\equiv i\pmod{t}$.

 Let us split this path $P_{2atn}$ into consecutive paths $Q_1, \ldots, Q_{2an}$ each on $t$ vertices.
 More precisely, we let $Q_i:=(x_{(i-1)t+1},\ldots,x_{it})$ for $i=1,\ldots,2an$.
The following auxiliary graph plays an important role in our proof: 
\begin{align*}
&\text{Let $H'$ be a graph on $V(H')=\{Q_1, \ldots, Q_{2an}\}$ such that $Q_iQ_j\in E(H')$ if and only if}\\
&\text{there is an edge in $G$ between the vertex sets of $Q_i$ and $Q_j$}.
\end{align*}

\begin{claim}\label{claim:Hp}
$H'\in \cP(2a,tB,C,t,\delta,n)$.
\end{claim}
\begin{claimproof}
We verify the four conditions of Definition~\ref{def:P}.
Since $H'$ has $2an$ vertices, condition~\ref{def:P-1} holds.
Since $\Delta(G)\leq B$ and for any $Q_i\in V(H')$ we have $|Q_i|=t$ (as a subset of $V(G)$), there are at most $tB$ edges in $G$ with an endpoint in $Q_i$.
Then, $\Delta(H')\leq tB$, so that condition~\ref{def:P-2} holds. 

To show that condition~\ref{def:P-3} holds we need to prove that there exists $f_{H'}>0$ such that, for all pairs of disjoint subsets $X$, $Y\subseteq V(H')$ of size $Cn$, we have $d_{H'}(X,Y)=(1\pm\eps)f_{H'}$.
Let $X,Y\subseteq V(H')$ be disjoint subsets both of size~$Cn$.
Without loss of generality let $X=\{Q_1,\ldots,Q_{Cn}\}$ and let $Y=\{Q_{Cn+1},\ldots, Q_{2Cn}\}$. 
Let ${X_G}=\bigcup_{j=1}^{Cn}Q_j\subseteq V(G)$ and ${Y_G}=\bigcup_{j=C n+1}^{2Cn}Q_j\subseteq V(G)$.
Since $|X_G|=|Y_G|=tCn$ and $t\geq 1$, we have
\begin{equation}\label{eq:XY_G}
|X_G|,|Y_G|\geq Cn.
\end{equation}
Then, since $G\in\cP(A,B,C,T, \delta, n)$, there is exists $f_G>0$ such that
\begin{equation*}
e_G({X_G},{Y_G})=(1\pm\delta)f_G\cdot(tCn)^2.
\end{equation*}
Since $G\in\cP(A,B,C,T,\delta,n)$, there is no cycle of length at most $2T$ in $G$.
It follows that there is at most one edge between any pair $Q_i, Q_j\subseteq V(G)$ else we create a cycle of length at most $2t\le2T$ in $G$.
It follows that 
\begin{equation}\label{eq:edgeH}
e_{H'}({X},{Y})=e_G({X_G},{Y_G})=(1\pm\delta)f_G\cdot(tCn)^2,
\end{equation}
and thus 
\begin{equation}\label{eq:densH}
d_{H'}({X},{Y})=(1\pm\delta)f_G\cdot t^2\, .
\end{equation}
So~\ref{def:P-3} follows by letting $f_{H'}=t^2f_G$.

It remains to verify condition~\ref{def:P-4}.
Recall that any vertex of $H'$ corresponds to a path of on~$t$ vertices in $G$.
Thus, a cycle of length at most $2t$ in $H'$ implies the existence of a cycle of length at most $2t^2$ in $G$.
Since $2T\geq 2t^2$ and $G$ has no cycles of length at most $2T$, we conclude that $H'$ contains no cycle of length at most $2t$, which completes the proof of Claim~\ref{claim:Hp}.
\end{claimproof}

Observe that, since $t\geq 2$, equation~\eqref{eq:XY_G} also holds if $X$ and $Y$ are subsets of $V(H')$ with $|X|=|Y|=Cn/2$.
Therefore,~\eqref{eq:densH} holds for such subsets, from where we conclude that for any pair of subsets $X'$, $Y'\subset V(H')$ with $|X'|=|Y'|=Cn/2$ we have $e_{H'}(X',Y')\geq 1$.
Then, by Proposition~\ref{cl:edgeboost} applied to $H'$ with $\alpha=2a$, $\beta=C$ and $\mu=C/2$, we have
\begin{equation*}
e_{H'}(X,Y)\geq Cn\text{ for any pair $X,Y\subset V(H')$ with $|X|=|Y|=Cn$}.
\end{equation*}
Together with~\eqref{eq:edgeH} and the choice of $C$, we obtain (recalling that $f_{H'}=t^2f_G$)
\begin{equation}\label{eq:pless1}
f_{H'} > \frac{1}{2Cn} \ge \frac{120a}{\eps^2 c^2 n}.
\end{equation}

\begin{claim}\label{cl:new}
There exists $H \subset H'$ such that $H\in \cP(a,b,c,t,\eps,n)$.
\end{claim}

\begin{claimproof}

We first obtain $H''\subseteq H'$ by picking each edge of $H'$ with probability $p=120a/(\eps^2 c^2 f_{H'} n)$ independently at random.
Note that \eqref{eq:pless1} guarantees that $p<1$.

For disjoint subsets $X,Y \subseteq V(H'')$ of cardinality $cn$ we have: 
\begin{equation*}
\mathbb{E}(d_{H''}(X,Y))=(1\pm\delta)f_{H'}\cdot p\,.
\end{equation*}
Therefore, using the union bound and Chernoff's inequality (Theorem~\ref{thm:chernoff}), we see that with probability at least
\begin{equation}
\label{chern}
1-\binom{an}{cn}^2\cdot 2 \cdot e^{-(\eps/3)^2 f_{H'}p(1-{\delta})c^2n^2/3}\,,
\end{equation}
every pair of sets of size $cn$ in $H''$ has density in the range $(1 \pm \tfrac{\eps}{3})(1\pm{\delta})f_{H'}\cdot p$.
Using that (\ref{chern}) is positive from our choice of $p$ and that 
$(1+\tfrac{\eps}{3})(1+{\delta})\leq 1+\eps$ and $(1-\tfrac{\eps}{3})(1-{\delta})\geq 1-\eps$, 
we may fix $H''\subseteq H'$ such that 
\begin{equation*}
d_{H''}(X,Y)=(1\pm\eps)f_{H'}\cdot p\, , 
\end{equation*}
for all disjoint subsets $X,Y\subseteq V(H'')$ of size $cn$. 

Note that by Proposition~\ref{prop:density}, it follows that
\begin{equation*}
|E(H'')| \leq (1+\eps)f_{H'}p \binom{2an}{2}\leq \frac{264 a^3 n}{\eps^2 c^2 }\leq abn\,,
\end{equation*}
where the last inequality follows from the fact that $(a,b,c,\eps)$ is a good quadruple.

We construct $H$ from $H''$ by sequentially removing the $an$ vertices of highest degree.
Since $H''$ has at most $abn$ edges we cannot have removed more than this many when creating $H$ and so $H$ has maximum degree at most $b$.
Since $H$ is a subgraph of $H'$, and $H'$ does not contain cycles of length at most $2t$, the same holds for $H$.
Finally, since deleting vertices preserves the property that pairs of subsets of size $cn$ have the correct density, we conclude that $H$ has property $\cP(a,b,c,t,\eps,n)$ as required.
\end{claimproof}

Recall that $J=(G[W])^R$ is the $R$th power of the subgraph of $G$ induced by the vertices of a subset $W\subseteq V(G)$ with $|W|\geq An/s$, where there is a collection of blue cliques $\{B(w)\colon w\in W\subseteq V(G)\}$ in our original colouring $\chi$ of $G^R\{T\}$ (see the beginning of the proof of Lemma~\ref{lem:indstep}).

Without loss of generality assume that $V(H)=\{Q_1, \ldots, Q_{an}\}$.
Furthermore, recall that $Q_i=(x_{(i-1)t+1},\ldots,x_{it})$ for $i=1,\ldots,an$, where the vertices $x_i$ belong to $J''=G[V(J')]$.
In what follows, when considering the graph $H^r(t)$, the $t$-clique corresponding to $Q_i$ is composed of the vertices $x_{(i-1)t+1},\ldots,x_{it}$, and hence $V\big(H^r(t)\big)\subseteq V(J'')\subseteq V(J)$.

\begin{claim}\label{claim:Hblow}
$H^r(t)\subseteq J$.
Moreover, $J$ contains a copy of $H^r\{t\}$ that is monochromatic in grey.
\end{claim}
\begin{claimproof}

We will prove that $H^r(t)\subseteq J$ where $Q_1, \ldots, Q_{an}\subseteq{V(J)}$ are the $t$-cliques of $H^r(t)$.
Suppose that $Q_i$ and $Q_j$ are at distance at most $r$ in the graph $H$.
Without loss of generality let $Q_i=Q_1$ and $Q_j=Q_m$ for some $m\leq r$.
Moreover, let $(Q_1,Q_2,\ldots, Q_m)$ be a path in $H$.
Note that there exist vertices $u_1,\ldots, u_{m-1}$ and $u_2',\ldots, u_m'$ in $G[W]$ such that $u_1\in Q_1$, $u_m'\in Q_m$, $u_j, u_j'\in Q_j$ for all $j=2,\ldots, {m-1}$ and $\{u_i, u_{i+1}'\}$ is an edge of $G[W]$ for $i=1,\ldots, m-1$.

Let $u_1'\in Q_1$ and $u_m\in Q_m$ be arbitrary vertices.
Since for any $j$, the set $Q_j$ is spanned by a path on $t$ vertices in $G[W]$, it follows that $u_j$ and $u_j'$ are at distance at most $t-1$ in $G[W]$ for all $1\leq j\leq m$.
Therefore, $u_1'$ and $u_m$ are a distance at most $(t-1)m+(m-1)< tr\leq R$ in $G[W]$ and hence $u_1'u_m$ is an edge in $J=(G[W])^R$.
Since the vertices $u_1'$ and $u_m$ were arbitrary, we have shown that if $Q_i$ and $Q_j$ are adjacent in $H^r$ (i.e., $Q_i$ and $Q_j$ are at distance at most $r$ in $H$) then $(Q_i,Q_j)$ gives a complete bipartite graph $B(Q_i, Q_j)$ in $J$. 
Moreover, taking $i=j$ we see that each~$Q_i$ in~$J$ must be complete. 
This implies that $H^r(t)$ is a subgraph of $J$.
 
 For the second part of the claim we consider what colours the edges of this copy of $H^r(t)$ in~$J$ can have. 
 Recall from the definition of $J'$ that we found vertex subsets $V_{1}, \ldots, V_{t}\subseteq J$ such that any edge of $J$ lying between different parts must be grey.
 Moreover each set $Q_i\subseteq J$ takes precisely one vertex from each set $V_{1}, \ldots, V_{t}$.
 It follows that each $Q_i$ must span a monochromatic grey clique in~$J$. 
 For the edges between $Q_i$ and $Q_j$, let $x\in Q_i$ and $y\in Q_j$. 
 If $x$ and $y$ lie in different parts of the partition $\{V_1,\ldots, V_t\}$ then the edge $\{x,y\}$ must be coloured grey.
 If $x$ and $y$ lie in the same part then we cannot determine the colour of $\{x,y\}$.
 In other words, all edges between $Q_i$ and $Q_j$ are grey except for possibly a matching. 
 We can therefore find a monochromatic grey copy of $H^r\{t\}$ in our copy of $H^r(t)$ in $J$, which completes the proof of Claim~\ref{claim:Hblow}.
\end{claimproof}

To complete the proof of Lemma~\ref{lem:indstep}, we will embed a copy of the graph $H^r\{t\}\subset J$ found in Claim~\ref{claim:Hblow} in $G^R\{T\}$ in such a way that $H^r\{t\}$ uses at most $s-1$ colours. 

\begin{claim}\label{claim:LLL}
$G^R\{T\}$ contains a copy of $H^r\{t\}$ with no blue edges.
\end{claim}

\begin{claimproof}
Recall that each vertex $u$ in $J$ corresponds to a clique $B(u)\subset G^R\{T\}$ of size $T'=s^{-1}\log_s T$ and that this clique is monochromatic in blue in the original colouring $\chi$ of $E(G^R\{T\})$.
Recall also that if an edge $\{u,v\}$ in $J$ is coloured grey in our auxiliary colouring $\chi'$, then there is no blue copy of $K_{2k,2k}$ in the bipartite graph between $B(u)$ and $B(v)$ in $G^R\{T\}$. By the K\H{o}v\'ari--S\'{o}s--Tur\'{a}n theorem (Theorem~\ref{thm:kst}), there are then at most $4T'^{2-1/2k}$ blue edges between $B(u)$ and $B(v)$. Recall further that $B(u)$ and $B(v)$ are, respectively, subcliques of the $T$-cliques $C(u)$ and $C(v)$ in $G^R\{T\}$. 
Since $\{u,v\}$ is an edge of $J$, there is a complete bipartite graph with a matching removed between $C(u)$ and $C(v)$ in $G^R\{T\}$ and so there is a complete bipartite graph with at most a matching removed for $B(u)$ and $B(v)$. 
It follows that there are at least
\begin{equation*}
(T')^2-T'-4(T')^{2-1/2k}
\end{equation*}
non-blue edges between $B(u)$ and $B(v)$. 

Using the grey copy of $H^r\{t\}\subseteq J$ obtained in Claim~\ref{claim:Hblow} as a `template', we will embed a copy of $H^r\{t\}$ in $G^R\{T\}$ with no blue edges. 
For each vertex $u\in V(H^r\{t\})\subseteq V(J)$ we will pick precisely one vertex from $B(u)\subseteq G^R\{T\}$ in our embedding.
The argument proceeds by the Lov\'asz Local Lemma.

For each $u\in V(H^r\{t\})\subseteq V(J)$ let us choose~$x_u\in B(u)$ uniformly and independently at random.  
Let~$e=\{u,v\}$ be an edge of our grey copy of $H^r\{t\}$ in $J$.
As pointed out above, we know that there are at least $T'^2-T'-4(T')^{2-1/2k}$ non-blue edges between $B(u)$ and $B(v)$. 
Letting~$A_e$ be the event that~$\{x_u,x_v\}$ is a blue edge or a non-edge in~$G^R\{T\}$, we have that
\begin{equation*}
\PP[A_e]\leq\frac{T'+4(T')^{2-1/2k}}{(T')^2}\leq5(T')^{-1/2k}.
\end{equation*}

  The events~$A_e$ are not independent, but we can define a dependency
  graph~$D$ for the collection of events~$A_e$ by adding
  an edge between~$A_e$ and~$A_f$ if and only
  if~$e\cap f\neq\emptyset$.
  Then, $\Delta:=\Delta(D)\leq2\Delta(H^r\{t\})\leq2(b^{r+1}t+t^2)$.
    Given that 
  \begin{equation*}
    \label{eq:4epd}
    4\Delta\PP[A_e]\leq40(b^{r+1}t+t^2)T'^{-1/2k}\leq1
  \end{equation*}
  for all~$e$, the Local Lemma tells us
  that~$\PP\big[\bigcap_{e}\bar A_e\big]>0$, and hence a
  simultaneous choice of the~$x_u$'s ($u\in V(H^r\{t\})$) is possible,
  as required.  This concludes the proof of Claim~\ref{claim:LLL}.
\end{claimproof}
The proof of Lemma~\ref{lem:indstep} is now complete.
\end{proof}


\end{document}